\documentclass[12pt]{amsart}
\usepackage{amssymb,amsmath,amsthm}

\usepackage[normalem]{ulem}

\def\e{{\rm e}}
\def\cic{\mathbf}
\def\eps{\varepsilon}

\def\d{{\rm d}}

\def\R {\mathbb{R}}

\def\V {{\mathsf V}}

\def\D {{\mathcal D}}

\def\F {{\mathcal F}}
\def\sh {{\mathsf{sh}}}

\def\S{{\mathbf S}}
\def\size{{\mathrm{size}}}
\def\tops{{\mathrm{tops}}}
\def\T{{\mathbf{T}}}

\def\Q {{\mathsf Q}}
\def\tr {{\mathsf \Lambda}}
\def\M {{\mathsf M}}

\def\Z {{\mathbb Z}}
\def \l {\langle}
\def \r {\rangle}

\def \and{\qquad\text{and}\qquad}

\def \no#1#2#3 {{\bf #1} (#3), #2.}
\def \eds#1#2#3 {#1, #2, #3.}
\newtheorem{proposition}{Proposition}
\newtheorem{theorem} {Theorem}

\newtheorem{lemma}[proposition]{Lemma}
\theoremstyle{definition}

\newtheorem{remark}[proposition]{Remark}

\newtheorem{conjecture}[theorem]{Conjecture}

\numberwithin{proposition}{section}
\numberwithin{equation}{section}

\title[Endpoint bounds for the quartile operator]{Endpoint bounds for the quartile operator}

\author[C.\ Demeter]
{Ciprian Demeter}
\address{
Dept.\ of Mathematics
\newline\indent
Indiana University
\newline\indent
Bloomington, IN 47405 - USA}
\email{demeterc@indiana.edu {\rm (C.\ Demeter)} }

\author[F.\ Di Plinio]{Francesco Di Plinio}
\address{INdAM - Cofund Marie Curie Fellow at Dipartimento di Matematica, \newline \indent Universit\`a degli Studi di Roma ``Tor Vergata'', \newline  \indent Via della Ricerca Scientifica,   00133 Roma,  Italy   \newline \indent \centerline{and}   \indent
  The Institute for Scientific Computing and Applied Mathematics,
\newline\indent
Indiana University
\newline\indent
831 East Third Street, Bloomington, Indiana  47405, U.S.A. }
\email{diplinio@mat.uniroma2.it {\rm (F.\ Di Plinio)} }

\thanks{The first author is partially supported by a Sloan Research Fellowship and by NSF Grant DMS-0901208. The second author is an INdAM - Cofund Marie Curie Fellow and is  partially
supported by the National Science Foundation under the grant
   NSF-DMS-1206438, and by the Research Fund of Indiana University.}

\begin{document}

\begin{abstract} It is a result by Lacey and Thiele \cite{LT1,LT2} that the bilinear Hilbert transform maps $L^{p_1}(\R) \times L^{p_2}(\R) $ into $L^{p_3}(\R)$   whenever $(p_1,p_2,p_3)$ is a H\"older tuple with  $p_1,p_2 >  1 $ and $p_3>\frac23$. We study the behavior of the quartile operator, which is the Walsh model for the bilinear Hilbert transform, when $p_3=\frac23$. We show that the quartile operator maps  $L^{p_1}(\R) \times L^{p_2}(\R) $ into $L^{\frac23,\infty}(\R)$ when $p_1,p_2>1$ and one component is restricted to subindicator functions. As a corollary, we derive that the quartile operator maps $L^{p_1}(\R) \times L^{p_2,\frac23}(\R) $ into $L^{\frac23,\infty}(\R)$. We also provide weak type estimates and boundedness on Orlicz-Lorentz spaces near $p_1=1,p_2=2$ which improve, in the Walsh case,   the results of \cite{BG}, \cite{CGMS}. Our main tool is the multi-frequency Calder\'on-Zygmund decomposition from \cite{NOT}.
\end{abstract}
\maketitle

\section{The Walsh phase plane and the quartile operator}
\subsection{The Walsh phase plane}  Let
$$
W_0(x) = \cic{1}_{[0,1]}(x), \qquad W_n(x)= \prod_{k\geq 0} \big(\mathrm{sign}\sin(2^{k }  2\pi x) \big)^{\eps_k}, \; n \geq 1,
$$
be the Walsh system on $[0,1]$. We have written    $n=\sum_{k\geq 0} \eps_k 2^k$ in binary representation,   that is, $\eps_k\in \{0,1\}$. The system $\{W_{n}:n\geq 0\}$ is an orthonormal basis of $L^2(0,1)$.

We will denote by $\D$ the standard dyadic grid on $\R$,
and  write $\ell(I)$ for the left endpoint of a dyadic interval $I$. A \emph{tile} $s=I_s \times \omega_s \subset \R \times \R_+$ is a dyadic rectangle of area 1, that is, $I_s$ and $\omega_s$ both belong to $\D$ and   $|I_s||\omega_s|=1$ . We define the corresponding Walsh wave packet
$$
w_s(x) = \mathrm{Dil}^{2}_{|I_s|} \mathrm{Tr}_{\ell(I_s)} W_{n_s} (x)=|I_s|^{-1/2}W_{n_s}\Big(\frac{x-\ell(I_s)}{|I_s|}\Big), \qquad n_s:= |I_s| \ell(\omega_s).
$$
Let $\mathbf{Ti}$ be the set of all tiles. It is immediate to see that $\|w_s\|_2=1$ for all $s \in \mathbf{Ti}$. Also, the following fundamental orthogonality property holds:
\begin{equation}
\label{walshorth}
s \cap s' = \emptyset\implies \l w_s, w_{s'} \r = 0.
\end{equation}
  A \emph{quartile} $s=I_{s} \times \omega_{s}$ is a dyadic rectangle with $|\omega_s|= 4|I_s|^{-1}$.
 We think of $\omega_s$ as the union of its dyadic grandchildren $\omega_{s_j}$, $j=1,\ldots,4$, with $\ell(\omega_{s_1})< \cdots <\ell(\omega_{s_4})$, and  denote by $s_j$ the tiles $I_{s} \times \omega_{s_j}$, which we call \emph{frequency grandchildren} of $s$. The set of all quartiles will be denoted by $\mathbf{Qt}$.

 \subsection{The quartile operator}We define the quartile operator
$$
\V_\S (f_1,f_2)(x) = \sum_{s \in \S} \frac{1}{\sqrt{|I_s|}} \l f_1, w_{s_1} \r \l f_2, w_{s_2} \r w_{s_3}(x),
$$
and the associated trilinear form
\begin{equation} \label{trf}
\tr_\S(f_1,f_2,f_3) = \l\V_\S (f_1,f_2), f_3 \r,
\end{equation}
where $\S \subset \mathbf{Qt}$ is an arbitrary collection of quartiles. The quartile operator has been introduced in  \cite{Th1} (see also \cite{Ththesis}) as a discrete model for the bilinear Hilbert transform
\begin{equation}
\label{BHT}
\mathsf{BHT}(f_1,f_2)(x) = \mathrm{p.v.} \int_{\R} f_1(x-t) f_2(x+t) \,\frac{\d t}{t}.
\end{equation}
Hereafter, $f^*:\R^+ \to [0,\infty)$ indicates the decreasing rearrangement of $f$, and  for $0 <p,q\leq \infty$, $p\not=\infty$, we denote by $L^{p,q}(\R)$ the Lorentz space with norm (or quasi-norm if either $p$ or $q$ are less than 1) 
$$
\|f\|_{ {p,q} } :=   \Big\| {t^\frac{1}{p}} f^*(t)\Big\|_{L^q\left(\R_+, \frac{\d t}{t}\right)}.
 $$
Let $(p_1,p_2,p_3) \in (0,\infty]^3$ be a H\"older tuple of exponents, that is
$$
\textstyle \frac{1}{p_1} + \frac{1}{p_2} = \frac{1}{p_3}.
$$
 A dilation-invariance argument shows that if $p_i, q_i \in (0,\infty],$ $i=1,2,3$ are tuples of exponents such that
\begin{equation}
 \label{BHTbd}\mathsf{BHT}: L^{p_1,q_1}(\R) \times L^{p_2,q_2}(\R) \to L^{p_3,q_3}(\R)
\end{equation}
then necessarily  $(p_1,p_2,p_3)$ is a H\"older tuple.   
It is a celebrated result by Lacey and Thiele \cite{LT1,LT2} that
\begin{equation}
 \label{BHTbdnonLor}
\mathsf{BHT}: L^{p_1}(\R) \times L^{p_2}(\R) \to L^{p_3}(\R)
\end{equation}
whenever $(p_1,p_2,p_3)$ is a H\"older tuple with  $p_1,p_2 >  1 $ and $p_3>\frac23$. Their proof uses the fact that $\mathsf{BHT}$ is an average of Fourier analogues of  $\V_\mathbf{Qt}$ (where the Walsh wave packets are replaced by certain ``Fourier wave packets") that are  uniformly bounded. The analysis of Walsh model sums is technically simpler than the one of their Fourier counterparts (in part due to \eqref{walshorth}), but it generally preserves the main conceptual difficulties from the Fourier world. Due to this, the ``Walsh case" is the ideal scenario for conveying new ideas in the most transparent way.

The quartile operator shares the same H\"older scaling property   exhibited by $\mathsf{BHT}$ and is known to be bounded in the same range of exponents, \cite{Th1}.
 A counterexample by Lacey and Thiele \cite{LT2}, see \cite{Lac} for details,  shows that  having $p_3\geq\frac23$ is necessary  in order for bounds of the type \eqref{BHTbd} to hold true for (the Fourier analogue of) $\V_\mathbf{Qt}$.

 The main purpose of the present paper is the investigation of the endpoint behavior of $\V_\mathbf{Qt}$. It is natural to conjecture
\begin{conjecture}For each $p_1,p_2 >1$ such that $\frac{1}{p_1} + \frac{1}{p_2} = \frac32$ we have
\label{conj}
\begin{equation}
\label{cjt}
\V_\mathbf{Qt}: L^{p_1}(\R) \times L^{p_2}(\R) \to L^{\frac23,\infty}(\R).
\end{equation}

\end{conjecture}
We expect the same bound to hold for $\mathsf{BHT}$.
To support the conjecture we mention that there is no known counterexample to disprove even the strong type bound
 $$\V_\mathbf{Qt}:L^{p_1}(\R) \times L^{p_2}(\R) \to L^{\frac23}(\R).$$

\subsection{A summary of previous results} Perhaps the simplest way to prove \eqref{BHTbdnonLor}  is to  establish the so called generalized restricted type bounds of the form
\begin{equation} \label{rwtintro}
|\tr_\mathbf{Qt}(f_1,f_2,f_3)| \lesssim  \prod_{j=1}^3 |F_j|^{\alpha_j}, \qquad |f_1| \leq \cic{1}_{F_1}, \;|f_2| \leq \cic{1}_{F_2},\;  |f_3| \leq \cic{1}_{F_3'} \end{equation}
 where $F_3' \subset F_3$ is an appropriately chosen major set, for all  tuples $(\alpha_1, \alpha_2,\alpha_3)$ satisfying
\begin{equation}
\label{hexa-tuples}
\textstyle 1>\alpha_j \geq 0,\; j=1,2, \qquad 1>\alpha_3> -\frac12,\qquad \sum_{j=1}^3 {\alpha_j}=1.
\end{equation}
Here and in what follows, the term \emph{major set} indicates any subset $F'\subset F$ of a set $F$ of finite measure with $|F'| \geq \frac18 |F|.$
These estimates are then turned into strong type bounds via interpolation.

We also mention some partial results towards the Conjecture \ref{conj}.
A refinement of the proof by Lacey and Thiele by Bilyk and Grafakos \cite{BG} provides  the following logarithmically bumped up version of \eqref{rwtintro} near the endpoint tuple $(1,\frac12,-\frac12)$: \begin{equation} \label{BGest}
|\tr_\mathbf{Qt}(f_1,f_2,f_3)| \lesssim  |F_1| |F_2|^{\frac12} |F_3|^{-\frac12} \log\big( \e + \textstyle \frac{|F_3|^2}{|F_1||F_2|}\big)^2, \end{equation} for   functions $|f_1| \leq \cic{1}_{F_1}, \,|f_2| \leq \cic{1}_{F_2},\,  |f_3| \leq \cic{1}_{F_3'}$,
in the regime (say) $|F_1| \leq |F_2|.$
 The analogous estimate for $\mathsf{BHT}$ is   in turn used to derive distributional estimates, and, in the subsequent work \cite{CGMS}, to derive estimates on Lorentz and Orlicz-Lorentz spaces near the endpoint  $p_3= \frac23.$
\section{Main results} \label{main}
We now detail our first main theorem, which is a weak $L^{\frac23}$ bound with only one subindicator function.
\begin{theorem} \label{mainth}Let $p_1,p_2 >1, \, \textstyle\frac{1}{p_1}+ \frac{1}{p_2}= \frac32.$ Then we have the mixed  type estimate
\begin{equation} \label{mrwtmain}
\lambda \big|\big\{x: | \V_\mathbf{Qt} (f_1,f_2)(x)| > \lambda\}\big|^{\frac32} \lesssim (p_1)'\|f_1\|_{p_1} |F_2|^{\frac{1}{p_2}}, \qquad \forall\; |f_2|\leq \cic{1}_{F_2}.
\end{equation}
\end{theorem}
Theorem \ref{mainth} is a fairly direct consequence of the proposition below, which is of independent interest. Indeed, it is worth mentioning  that the generalized restricted type estimate \eqref{rwtintro} was only known to hold in the open range $\alpha_3> -\frac12$. This is because --in the terminology of the following sections-- the summation of the forests gives rise to a lacunary series that diverges when $\alpha_3= -\frac12$.
 The following proposition  implies  the endpoint $\alpha_3= -\frac12$ generalized restricted type  \eqref{rwtintro}, but in fact it is stronger than that.  The proof will be carried out   in Section \ref{sect5};  we kept  track of the explicit dependence on $p_1$ of the implied constant in the estimate  \eqref{mrwt}.\begin{proposition} \label{mainprop}
Let $p_1,p_2 >1, \,\textstyle\frac{1}{p_1}+ \frac{1}{p_2}= \frac32.$ Then for every $f_1 \in L^{p_1}(\R)$ and every $F_2, G_3 \subset \R$, there exists a major set $F_3 \subset G_3$ such that
\begin{equation} \label{mrwt}
|\tr_\mathbf{Qt}(f_1,f_2,f_3)| \lesssim (p_1)' \|f_1\|_{p_1} |F_2|^{\frac{1}{p_2}} |G_3|^{-\frac{1}{2}}, \qquad \forall\; |f_j| \leq \cic{1}_{F_j}, \; j=2,3.
\end{equation}
\end{proposition}
\begin{proof}[Proof that Proposition \ref{mainprop} implies Theorem \ref{mainth}] Let $f_1, F_2$, $f_2$ with $|f_2| \leq \cic{1}_{F_2} $ be given. By using a standard approximation argument, it suffices to prove the bound for an arbitrary finite subset $\S\subset\mathbf{Qt}$, as long as the bound is independent of $\S$. Since $\S$ is a finite collection, there exists a constant $C_0$, possibly depending on $p_1,\S,f_1, $ and $f_2$, such that \eqref{mrwtmain} holds. We show that the constant is independent of all but $p_1.$ Fix $\lambda>0$ and take $G_3$ to be the set in the left hand side of \eqref{mrwtmain}. The existence of $C_0$ implies that $G_3$ has finite measure, so that we may apply Proposition \ref{mainprop}, and obtain the existence of a major set $F_3 \subset G_3$ such that \eqref{mrwt} holds for any $f_3$ with $|f_3| \leq \cic{1}_{F_3} $. We can choose $f_3$ such that $|\tr_\S(f_1,f_2,f_3)|>\lambda \frac{|G_3|}{2}$, so that
$$
\lambda|G_3| \lesssim |\tr_\S(f_1,f_2,f_3)| \lesssim (p_1)'\|f_1\|_{p_1} |F_2|^{\frac{1}{p_2}} |G_3|^{-\frac{1}{2}},
$$
 which, rearranged, gives exactly \eqref{mrwtmain}.
\end{proof}
We are interested in extrapolating the bound of Theorem \ref{mainth} to Lorentz spaces. In this direction we have the  theorem below, which can be seen as a partial result towards the conjectured bound \eqref{cjt}.  We next  show   how this follows from Theorem \ref{mainth},  via  linear extrapolation in the style of \cite{CCM}.
\begin{theorem} \label{lorentzthm}Let $p_1,p_2 >1, \, \textstyle\frac{1}{p_1}+ \frac{1}{p_2}= \frac32.$ Then
\begin{equation} \label{lorenzest}
\|\V_\mathbf{Qt} (f_1,f_2)\|_{ {\frac{2}{3},\infty} } \lesssim_{p_1}\|f_1\|_{ p_1 } \|f_2\|_{{p_2,\frac23} }.
\end{equation}
\end{theorem}

\begin{remark} Note that ${L^{p_2,\frac23}(\R)}$ is a proper subspace of ${L^{p_2} (\R)}$; the result of Theorem \ref{lorentzthm} is therefore  a  strictly weaker version of \eqref{cjt}.\end{remark}
\begin{proof}[Proof that Theorem \ref{mainth} implies Theorem \ref{lorentzthm}]
 We can normalize $\|f_1\|_{p_1} =1 = \|f_2 \|_{{p_2,\frac23}}$.  We perform the well-known decomposition
$$
f_2 = \sum_{k \in \Z} 2^{k} g_{k}, \qquad  |g_k| \leq \cic{1}_{G_k}, \; k \in \Z,
$$
with  each $G_k$ being a bounded set and
\begin{equation}
\label{esteq}
\big\|\{2^{k}|G_k|^{\frac{1}{p_2}}\}\big\|_{\ell^{\frac23}} \sim_{p_1} 1 = \|f_2 \|_{{p_2,\frac23}}.
\end{equation} We then have
$$
h:=\V_{\S} (f_1,f_2)= \sum_{k \in \Z}c_kh_k, \qquad  h_k=\frac{\V_\S (f_1,g_k) }{\|\V_\S (f_1,g_k)  \|_{\frac23,\infty}}, \quad c_k=2^{k} \|\V_\S (f_1,g_k)  \|_{\frac23,\infty}.
$$
Applying Theorem \ref{mainth}, we obtain the estimate \begin{equation}
\label{estcoeff} |c_k| \lesssim 2^{k}\|f_1\|_{p_1} |G_k|^{\frac{1}{p_2}}= 2^{k}  |G_k|^{\frac{1}{p_2}}.
\end{equation}
It is shown in \cite[Theorem 3.1]{CCM} (see also \cite[Theorem 2.1]{CGMS}) that
$$
\|f \|_{L^{q, r}} \lesssim \inf \bigg\{ \|\{\beta_k\}\|_{\ell^{q}} : f= \sum_{k} \beta_k \phi_k, \; \|\phi_k\|_{L^{q,r} }\leq 1  \bigg\}, \quad \forall \;0<q\leq 1, 0<r \leq \infty.
$$
Therefore, taking advantage of \eqref{estcoeff} and subsequently of \eqref{esteq},
$$
\|h \|_{ {\frac23, \infty}}  \lesssim \|\{c_k\}\|_{\ell^{\frac23}} \leq \big\|\{2^{k}|G_k|^{\frac{1}{p_2}}\}\big\|_{\ell^{\frac23}} \sim_{p_1} 1,
$$ and this completes the proof.\end{proof}

 Theorems \ref{mainth} and \ref{lorentzthm} do not cover the case $p_1=1,p_2=2$. In this subsection, we derive endpoint results involving Orlicz-Lorentz substitutes of $L^1(\R)$ and $L^{\frac23,\infty}(\R)$  as a consequence of the   estimate
\begin{align} \label{logproprearr}  &
\forall 
\;  f_1 \in L^1(\R)  \textrm{ with }   \|f_1\|_\infty \leq 1 ,\,  f_2 \in L^2(\R) ,  
  \\ & \big(\V_\mathbf{Qt} (f_1,f_2)\big)^*(t)  \leq \frac{\|f_1\|_1}{t^{\frac32}} \log\Big( \e +  \frac{t}{\|f_1\|_1} \Big)\|f_2\|_2, \qquad \forall t >0.\nonumber\end{align}
Let us show how to obtain \eqref{logproprearr} from the (family of) weak type estimates contained in the following proposition, whose proof we defer to Section \ref{sect5}.
\begin{proposition} \label{logprop} Let $1<p<2$, and $r$  defined by $\frac1p+\frac12=\frac1r$. Then  $\|\V_\mathbf{Qt}\|_{L^p \times L^2 \to L^{r,\infty}} \lesssim p'$, that is
\begin{equation} \label{wtelog}
\lambda\big|\big\{x: |\V_\mathbf{Qt} (f_1,f_2)(x)| > \lambda\big\}\big|^{\frac1r} \lesssim p'\|f_1\|_{p} \|f_2\|_2.\end{equation}
\end{proposition}
Note that, for any given $1<p<2$,  \eqref{wtelog}  can be rewritten in terms of the decreasing rearrangement $V(t):=\big(\V_\mathbf{Qt} (f_1,f_2)\big)^*(t)$ as 
\begin{equation} \label{rearr1}
 V(t) \lesssim p'\frac{\|f_1\|_p}{t^{\frac12+\frac1p}} \|f_2\|_2, \qquad \forall t>0.   
\end{equation}
Since $\|f_1\|_\infty \leq 1$, we can estimate $\|f_1\|_p\leq \|f_1\|_1^\frac{1}{p}$ in each instance of \eqref{rearr1}. This yields
$$
V(t) \lesssim  \|f_2\|_2  \frac{\|f_1\|_1}{t^{\frac32}} p' \Bigg( \frac{t}{\|f_1\|_1}\Bigg)^{\frac{1}{p'}}, \qquad 1<p<2.
$$
At this point, \eqref{logproprearr} follows by taking infimum over $1<p<2$ in the above inequality.
\begin{remark} In the literature, it is often the case that estimates in the vein of \eqref{logproprearr} are derived by first establishing the weaker version where one or both functions $f_1,f_2$ are restricted to be subindicator. This restriction is then lifted by using further properties of the (multi)-sublinear operator (e.g. $\eps$-$\delta$ atomic approximability). This is the approach adopted in \cite{CGMS} for the bilinear Hilbert transform and by Antonov \cite{ANTO} and Arias de Reyna \cite{ADR} for the Carleson maximal partial Fourier sum operator. Our own approach of extrapolating from  weak type estimates (like \eqref{wtelog}), as opposed to restricted weak type estimates as usual, bypasses the above additional step.  \end{remark}

 As in \cite{CGMS}, we introduce the weak type weighted Lorentz space $X $ with quasi-norm
$$
\|f\|_X= \sup_{t>0} \frac{t^{\frac32}}{\log (\e + t)} f^*(t).
$$
The space $X$ is a logarithmic enlargement of $L^{\frac23,
\infty}$. 
Note that an application of the trivial inequality $\frac12\log(\e+ab) \leq \log(\e+ a)\log(\e+b)$  for all $a,b>0$,  turns \eqref{logproprearr}, after rearranging and taking supremum over $t>0$, into
\begin{equation} \label{Xbound1}
 \|\V_{\mathbf{Qt}}(f_1,f_2)\|_X \lesssim  \|f_1\|_{1} \log \left( \e +\frac{1}{\|f_1\|_1} \right)\|f_2\|_2, 
 \qquad \forall \;f_1 \in L^1(\R)  \textrm{ with }   \|f_1\|_\infty \leq 1.
\end{equation}
The above estimate \eqref{Xbound1} can be coupled with linear extrapolation in the first function $f_1$ to obtain the two endpoint theorems below. 
\begin{theorem} \label{logthm1}  There holds $$
 \V_\mathbf{Qt}: L^{1,\frac23} \log L^{ \frac23}(\R) \times L^{2 }  (\R) \to  X,
$$
where $ L^{1,\frac23} \log L^{ \frac23}(\R) $ is the Lorentz-Orlicz quasi-Banach space with quasinorm
$$
\|f\|_{L^{1,\frac23} \log L^{ \frac23} } :=  \Big\| t  \log\big( \textstyle \e+ \frac1t\big)f^*(t)\Big\|_{L^{\frac23}\left(\R_+, \frac{\d t}{t}\right)}
$$
\end{theorem}
\begin{theorem} \label{logthm2}  We have that, for each $\eps>0$, $$
 \V_\mathbf{Qt}: L\log L (\log\log L)^{\frac12} (\log\log\log L)^{\frac12+\eps} (\R) \times L^{2 }  (\R) \to  X,
$$
where $  L \log L (\log\log L)^{\frac12} (\log\log\log L)^{\frac12+\eps}(\R)  $ is the Orlicz space\footnote{Given a Young's function $\varphi:[0,\infty] \to [0,\infty)$, the Orlicz (Banach) space $L^\varphi(\R)$ is the space of measurable functions on $\R$   with finite Orlicz norm
$$
\|f\|_{L^\varphi}:= \inf\Big\{\lambda>0 : \int_\R \textstyle\varphi\big(\frac{|f(x)|}{\lambda}\big)\, \d x \leq 1\Big\}.
$$}
with  Orlicz function
$$
 \varphi(t) = t \log \left( \e+ t \right) \big(\log  \log (\e^\e+t) \big)^{\frac12}\big(\log  \log  \log(\e^{\e^\e}+t )\big)^{\frac12+\eps}.
$$
\end{theorem}
Theorem 4 is obtained by fixing $f_2$ with $\|f_2\|_2=1$ and straightforwardly applying the linear extrapolation theorem \cite[Theorem 2.1(b)]{CGMS} to the linear operator $f_1 \mapsto \V_{\mathbf{Qt}}(f_1,f_2)$, in view of estimate \eqref{Xbound1}. Even the weaker version of \eqref{Xbound1} where $f_1$ is taken to be a subindicator function would comply with the assumption thereof. Theorem 5 follows from (the full strength of) \eqref{Xbound1} via the level set decomposition argument of \cite[Example 3.12]{CGMS}, again applied to   the linear operator $f_1 \mapsto \V_{\mathbf{Qt}}(f_1,f_2)$.
\begin{remark}  Theorems \ref{logthm1} and \ref{logthm2} improve  respectively 
 the bounds
\begin{align*} &
V_\mathbf{Qt}: L^{1,\frac23} \log L^{ \frac43}(\R) \times L^{2,\frac23} \log L^{ \frac43}(\R)   \to  X, \\ & V_\mathbf{Qt}: L (\log L)^2 (\log\log L)^{\frac12} (\log\log\log L)^{\frac12+\eps}(\R) \times L^{2,\frac23} \log L^{ \frac43}(\R) \to X
\end{align*} proven in Section 4.1 of \cite{CGMS} (when restricted to the Walsh case).  In particular, we can always take one function in $L^2(\R)$. We also stress that \eqref{Xbound1} allows us to freeze the $L^2$ function and only employ  \emph{linear} extrapolation results to derive our Theorems \ref{logthm1} and \ref{logthm2}, in contrast with the approach of \cite{CGMS} where bilinear extrapolation is needed.
Finally, it is easy to see that the two substitutes for $L^1(\R)$ in Theorems \ref{logthm1} and \ref{logthm2} are not comparable; thus none of the two results can be obtained from the other. 
\end{remark}
\begin{remark} One might be interested in finding an Orlicz-Lorentz modification of the pair $L^1(\R) \times L^2(\R)$ which is mapped by $\V_\mathbf{Qt}$ into $L^{\frac23, \infty} (\R)$ (strictly smaller than the space $X$ appearing in Theorems \ref{logthm1} and \ref{logthm2}).
This can be done  
 by means of the following  restricted weak type estimate: given any two sets $F_1,F_2 \subset \R$ with $|F_1| \leq |F_2|$,
\begin{equation}
\lambda \big|\big\{x: \V_\mathbf{Qt} (f_1,f_2)(x) > \lambda\}\big|^{\frac32} \lesssim |F_1| |F_2|^{\frac12}    \log\big( \e+\textstyle\frac{|F_2|}{|F_1|}  \big) , \label{extrwr} \end{equation} for all $|f_1|\leq \cic{1}_{F_1}, \; |f_2|\leq \cic{1}_{F_2}$;
estimate \eqref{extrwr} is obtained by applying \eqref{mrwtmain} with the optimal choice of $p_1,p_2$ given by $ (p_1)'= \log\big(\e+ \textstyle\frac{|F_2|}{|F_1|}  \big)$. Then an application of the multilinear extrapolation Theorem 2.6 in \cite{CGMS},  yields the bound
$$
\V_\mathbf{Qt}: L^{1,\frac23} \log L^{1,\frac23}(\R) \times L^{2,\frac23} \log L^{1,\frac23}(\R) \to L^{\frac23,\infty}( \R).
$$
\end{remark}

\section{Analysis and combinatorics in the Walsh phase plane}
We refer to \cite{Th1,Ththesis} for more details about the results in this section.

Let us   introduce some more notation.  To simplify the combinatorial arguments, it is convenient to split  $\mathbf{Qt}$ into two subsets $\mathbf{Qt}_1$, $\mathbf{Qt}_2$ such that
\begin{equation}
\label{sepscl}
s, s' \in \mathbf{Qt}_j, |I_s| > |I_{s'}| \implies |I_s| \geq 4 |I_{s'}| .
\end{equation}
For the rest of the paper  we will  assume to be working with finite collections of quartiles $\S \subset \mathbf{Qt}_1.$

As described in the introductory section, we denote by $s_j$ the tiles $I_{s} \times \omega_{s_j}$, which we call \emph{frequency grandchildren} of $s$. Symmetrically, we will denote by $s^j$, $j=1,\ldots,4$ the tiles $I_s^{j} \times \omega_s$, where $I_s^j$ are the four dyadic  grandchildren of $I_s$, and will refer to them as \emph{spatial grandchildren} of $s$. Finally, we use the notations
$$
 \S_{\star}^{\star} = \S_{\star}\cup \S^{\star}:=\{s_j: s \in \S, \, j=1, \ldots, 4\} \cup \{s^j: s \in \S, \, j=1, \ldots, 4\}.
$$
and
$$
\sh(\S):= \bigcup_{s \in \S} s \subset \R \times \R_+.
$$
for  the \emph{shadow} of  a set of tiles (or quartiles) $\S$ in the phase plane.

\subsection{Trees and phase space projection}

We  will use the well-known Fefferman order relation on  quartiles:
\begin{equation}
\label{feff}
s \ll s' \iff I_s \subset I_{s'}\text{ and } \omega_{s} \supset \omega_{s'}.
\end{equation}
Note that, as a consequence of \eqref{walshorth}, if two quartiles $s,s'$ are not related under $\ll$ then    $$ \sigma \in \{s\}^\star_\star, \; \sigma' \in\{s'\}^\star_\star    \implies \l w_{\sigma}, w_{\sigma'}\r =0. $$

A collection $\S\subset \mathbf{Qt}_1$ is called \emph{convex} if
\begin{equation}
s, s'' \in \S, \,s' \in \mathbf{Qt}_1 \,, s  \ll s' \ll s'' \implies s' \in  \S.
\end{equation}
We will use repeatedly that the intersection of two convex sets is convex.

 A collection of  quartiles $\T\subset \mathbf{Qt}_1$  is called \emph{tree} with top quartile $s_\T\in \mathbf{Qt}_1$ if $s\ll s_\T$ for all $s \in \T$. We use the notation $I_\T:=I_{s_\T}, \omega_{\T}=\omega_{s_\T}.$
We say that a tree $\T$ is a $j$-tree (for some $j \in \{1,\ldots,4\}$) if $\omega_{\T} \subset \omega_{s_j}$ for all $s \in \T \backslash \{s_\T\}$. Note that if $\T$ is a $j$-tree the tiles  $\{s_k: k \in \{1,\ldots,4\} \backslash\{j\}, s \in \T\}$ are pairwise disjoint.

If $\S$ is a finite collection of pairwise disjoint tiles, define
$$
H_\S:= \mathrm{span}\{w_s: s \in \S\}
$$ as a subspace of $L^2(\R)$.

The lemma below is geared towards the subsequent definition of  the phase-space projection $\Pi_{\T}$ associated to a convex $j$-tree $\T$.
\begin{lemma} \label{psp1}  Let $\T$ be a  finite, convex $j$-tree, with $j \in \{1,\ldots,4\}$. Then there exists a finite set of pairwise disjoint tiles  $\T' \subset \T_\star^\star$ such that $\sh( \T)  = \sh( \T')$ and with  the further property
\begin{equation}
\label{spacedisj}
s',s'' \in \T', \, I_{s'} \cap I_{s''}\neq \emptyset \implies s'=s_{j}, s''=s_k \textrm{\emph{
for some }} s \in \T,
\end{equation}
\end{lemma}
\begin{proof} We proceed by induction on the number of quartiles. The base case $\#\T=1$ is trivially true. Let now $\T$ be a given finite convex $1$-tree (to fix ideas). Choose a minimal (with respect to $\ll$) quartile $s \in \T$. Then $\dot \T:=\T \backslash \{s\}$ is a convex $1$-tree. By induction assumption, we find a collection of pairwise disjoint tiles $\dot\T' \subset \dot\T^\star_\star$ with $\sh(\dot \T')=\sh(\dot \T)$ and such that \eqref{spacedisj} holds. Let $I$ be the unique interval in $\{I_{s'}: s' \in \dot \T'\}$ such that $I_s \subsetneq I$. Since $\dot \T$ is a convex 1-tree, there exists a unique quartile $\sigma \in \dot \T$ such that $I_{\sigma}=I$ and $|I_\sigma|=4 |I_s|$. Moreover $\omega_{s_1} =\omega_{\sigma }$. Setting $$
\T'=  \big(\dot \T' \backslash \{\sigma\}_{\star} \big) \cup \{\sigma\}^\star \cup \{s\}_\star$$ we   see that $\sh(\T')=\sh(\T)$ and that $\T'$ satisfies \eqref{spacedisj}. The induction is complete.    \end{proof}
For   a finite convex $j$-tree $\T$, let  $\T'$ be the corresponding collection of pairwise disjoint tiles given by Lemma
\ref{psp1}. The wave packets $\{w_s: s \in \T'\}$ are an orthonormal basis of $H_{\T'}$, due to the corresponding tiles in $\T'$ being pairwise disjoint. Thus the orthogonal projection $\Pi_\T:L^2(\R) \to H_{\T'}$ can be written as
\begin{equation}
\label{psp}
\Pi_{ \T} f = \sum_{s \in \T'} \l f, w_s\r w_s.
\end{equation}
In particular, if $s$ is a quartile, then
$$\Pi_sf=\sum_{j=1}^4\l f, w_{s_j}\r w_{s_j}$$
\label{convlemma}
\begin{lemma} \label{psp2} Let $\T$ be a convex $j$-tree. We have the equality
\begin{equation}
\l f, w_s \r = \l\Pi_\T f, w_s \r \qquad \forall s \in \T^\star_\star.
\end{equation}
\end{lemma}
\begin{proof}  This follows immediately from the fact that $ s \in \T^\star_\star $ implies $w_s \in H_{\T'}$. In turn, this is a consequence of the recursion relation for Walsh wave packets. See \cite[Corollary 1.10]{Ththesis}. \end{proof}


\section{Sizes, trees and single tree estimates}
For an interval $I \subset \R$, we adopt the notation
\begin{equation}
\label{normLp} \|f\|_{L^p(I)}:= \bigg( \int_{I}|f(x)|^p \frac{\d x}{|I|}\bigg)^{\frac1p}, \qquad 1 \leq p \leq
\infty
\end{equation}
Accordingly, we define the dyadic maximal functions
\begin{equation}
\label{mp}  \M_p f(x) = \sup_{x \in I \in \D}  \|f\|_{L^p(I)}, \qquad 1 \leq p <\infty.
\end{equation}

Let $f \in L^2(\R)$  and let $\S$ be a set of quartiles. We set
$$
\size_{f} (\S) = \sup_{s \in \S}   \frac{\| \Pi_{s}f  \|_2}{\sqrt{|I_s|}}.
 $$
Note that$$
\size_{f} (\S) \sim \sup_{s \in \S^\star_\star}   \frac{|\l f, w_{s} \r |}{\sqrt{|I_s|}}.
 $$
It is immediate to see that
 \begin{equation}
\label{ubsize}
\size_f(\S) \leq \sup_{s\in \S} \inf_{x \in I_s} M_1f(x).
\end{equation}
\begin{lemma}Let $\T$ be a convex tree, $j=1, \ldots, 4$, and  $\T_j =\{s \in \T: \omega_{s_j} \supseteq \omega_{s_\T} \}$. We have the John-Nirenberg type inequality
 \begin{equation} \label{JN}
\big\|\Pi_{\T_j} f\big\|_p \leq 4 |I_\T|^{\frac1p} \size_f(\T), \qquad \forall\; 1 \leq p \leq \infty.
\end{equation}
\end{lemma}
\begin{proof} Since $\Pi_{\T_j} f$ is supported on $I_{\T}$, it suffices to show that
$$
\big\|\Pi_{\T_j} f\big\|_\infty \leq 4\, \size_f(\T).
$$
Note that ${\T_j}$ is also a convex tree. Fix  $x \in I_\T$. Due to \eqref{spacedisj}, we have that $x \in I_s$ for at most four tiles $s \in (\T_{j})'$. Thus
$$
|\Pi_{\T_j} f(x)| \leq 4   \sup_{s \in (\T_{j})' } |\l f, w_s \r| |w_s(x)| \leq 4\, \size_f (\T)
$$
as claimed. The lemma follows.
\end{proof}

The John-Nirenberg inequality above is all we need to produce an estimate on the trilinear form \eqref{trf} restricted to a tree.
\begin{lemma}Let $\T$ be a convex tree. Then
\begin{equation} \label{ste}
|\tr_\T(f_1,f_2,f_3)| \lesssim |I_\T|\Big( \prod_{j=1}^3 \size_{f_j}(\T) \Big)
\end{equation} \label{stelemma}
\end{lemma}
\begin{proof}
We argue separately for each $\T_j=\{s \in \T:   \omega_{\T} \subset \omega_{s_j}\} $. The argument for $s_\T$ is even simpler. Assume $j=1$ for simplicity. We have
\begin{align*}
|\tr_{\T_1}(f_1,f_2,f_3)| &= |\tr_{\T_1}\big(\Pi_{\T_1} f_1, \Pi_{\T_1} f_2,\Pi_{\T_1} f_3\big)| \\& \le \sum_{s \in\T_1 } |\frac{\l \Pi_{\T_1}f_1, w_{s_1} \r}{\sqrt{|I_s|}}  \l\Pi_{\T_1} f_2, w_{s_2} \r \l \Pi_{\T_1}f_3, w_{s_1} \r| \\ &\leq \Big( \sup_{s \in \T_1} \frac{|\l \Pi_{\T_1}f_1, w_{s_1} \r|}{\sqrt{|I_s|}} \Big)  \Big( \sum_{s \in \T_1}  |\l\Pi_{\T_1} f_2, w_{s_2} \r|^2  \Big)^{\frac12}  \Big( \sum_{s \in \T_1}  |\l \Pi_{\T_1} f_3, w_{s_3} \r|^2  \Big)^{\frac12}  \\& \leq \size_{f_1}(\T_1)\|\Pi_{\T_1} f_2\|_2 \|\Pi_{\T_1} f_3\|_2  \lesssim      |I_\T|\Big( \prod_{j=1}^3 \size_{f_j}(\T) \Big).
\end{align*}
We have used that $\{w_{s_j}: s\in \T_1\}$ is an orthonormal set for $j=2,3$, since the corresponding tiles $\{s_j:s \in \T_1\}$ are pairwise disjoint, and  \eqref{JN} to conclude.
\end{proof}A collection of quartiles $\S$ is called a \emph{forest}  if $\S$ can be partitioned in (pairwise disjoint) convex trees $\{\T: \T \in \F\}$. It may be that a given $\S$ may admit many such partitions $\F$. We define
$$
\tops(\S) = \inf_{\F}\sum_{\T \in \F} |I_\T|.
$$

The  next lemma skims the trees with large relative size off the collection $\S$, organizing them into a forest. \begin{lemma} \label{sizelemma}
Let $\S$ be a finite convex collection of quartiles with $\size_f(\S)\leq\sigma$. Then
$
\S= \S_{\mathsf{hi}} \cup \S_{\mathsf{lo}}
,$
such that \begin {itemize}
\item[$\cdot$]both $\S_{\mathsf{lo}}$ and $ \S_{\mathsf{hi}} $ are convex;
\item[$\cdot$]$\size_f(\S_{\mathsf{lo}} )\leq\frac\sigma2$;
\item[$\cdot$] $ \S_{\mathsf{hi}} $ is a forest with
$
\tops( \S_{\mathsf{hi}})  \lesssim \sigma^{-2}\|f\|^2_2.
$  \end{itemize}
\end{lemma}
\begin{proof} This is a recursive procedure. Set $\S_{\mathsf{stock}}:=\S$. Select a maximal (with respect to $\ll$) quartile  $t \in \S_{\mathsf{stock}}$ such that
\begin{equation} \label{ubsize2}
\|\Pi_t f\|_2 > \textstyle\frac{\sigma}{2} \sqrt{|I_t|}
\end{equation}
Define
$$
\T(t) = \{s \in \S_{\mathsf{stock}}: s \ll t\}.
 $$
and note that since $\S_{\mathsf{stock}}$ is convex, $\T(t)$ is a convex tree. Add $\T(t)$ to the family $\F$. Reset the new value $\S_{\mathsf{stock}}:=\S_{\mathsf{stock}}\setminus\{s:s\in \T(t)\}$, and restart the procedure.

The algorithm is over when there is no $t$ to be selected. Then define $ \S_{\mathsf{hi}} $ to consist of the union of  all the selected trees in $\F$, and $\S_{\mathsf{lo}}=\S\setminus\S_{\mathsf{hi}}$.

  The first two needed properties are easy to check.
 By maximality the selected quartiles $t$ are pairwise disjoint, and thus the functions $\Pi_tf$ are pairwise orthogonal, thanks to \eqref{walshorth}. It follows that
$$
\sum_{\T \in \F}|I_\T|=\sum_{t} |I_t| \leq 4 \sigma^{-2} \sum_{t} \|\Pi_tf\|_2^2  \leq 4\, \sigma^{-2}\|f\|^2_2.
$$
\end{proof}
We will decompose the model sum \eqref{trf} by organizing the quartiles of $\S$ into forests of definite size. The lemma below turns the tree estimate \eqref{ste} into a forest estimate.
\begin{lemma}   \label{lemmasfe}  Let $\S$ be a finite convex collection  of quartiles which is also a forest, with
$$
\size_{f_j}(\S)=\sigma_j, \;j\in\{1,2,3\}, \qquad
\tops(\S) \lesssim \sigma_2^{-2}\|f_2\|^2_2.
$$
Then
\begin{equation} \label{sferef}
|\tr_\S(f_1,f_2,f_3)| \lesssim  \sigma_3 \|f_1\|_2 \|f_2\|_2
\end{equation}
\end{lemma}
\begin{proof}
Choose the minimizing forest $\F$ and write $$\S=\bigcup_{\T\in\F}\T, \qquad  \textrm{with}\; \sum_{\T\in\F}|I_\T|\lesssim \sigma_2^{-2}\|f_2\|^2_2.$$
Let $n_0= -2\log \big(\frac{\sigma_2\|f_1\|_2}{\|f_2\|_2}\big)$. If $\sigma_1 \sigma_2^{-1}  \|f_2\|_2  \leq    \|f_1\|_2  $, in other words if $\size_{f_1}(\S)\le 2^{-\frac{n_0}{2}}$, then use
the single tree estimate \eqref{ste}  to get
$$
|\tr_\S(f_1,f_2,f_3)| \leq   \sum_{\T \in \F} |\tr_\T(f_1,f_2,f_3)| \lesssim   \tops(\S) \sigma_1 \sigma_2  \sigma_3 \lesssim \sigma_1 \sigma_2^{-1} \sigma_3 \|f_2\|_2^2\le \sigma_3 \|f_1\|_2 \|f_2\|_2.
$$

Otherwise, we have $\sigma_1 \geq 2^{-\frac{n_0}{2}}.$ We iteratively apply Lemma \ref{sizelemma} with respect to $f_1$, a finite number of times, and write
$$
\S= \tilde{\S}\cup(\cup_{n\leq n_0} \S_n), \quad \tops(\S_n) \lesssim 2^{n}\|f_1\|_2^2, \quad \size_{f_1}(\S_n) \leq 2^{-\frac{n}{2}}, \quad \size_{f_1}(\tilde{\S})\le 2^{-\frac{n_0}{2}}
$$
Note that we also continue to have
$$\size_{f_j}(\S_n), \size_{f_j}(\tilde{\S})\leq \sigma_j, \,j=2,3.
$$
Then, using  triangle inequality, \eqref{ste} again,  and summing up,
$$
|\tr_{\cup_{n\leq n_0} \S_n}(f_1, f_2,f_3)| \lesssim \sum_{n\leq n_0} 2^{\frac n2 }\sigma_2 \sigma_3 \|f_1\|_2^2 \leq 2^{\frac{n_0}{2}} \sigma_2 \sigma_3 \|f_1\|_2^2 \leq  \sigma_3 \|f_1\|_2 \|f_2\|_2,
$$as claimed.
To deal with the contribution from $\tilde{\S}$, redo the computations from the first case, this time with $\S:=\tilde{\S}$. Lemma \ref{sizelemma} guarantees $\tilde{\S}$ is convex. To see that it is a forest, note that it is the disjoint union of the convex trees $\tilde{\T}:=\T\cap \tilde{\S}$, with $\T\in\F$, each of which is assigned the same top as $\T$. It is then obvious that $\tops(\tilde{\S})\le \tops(\S)$.
\end{proof}

\section{Proofs of Propositions \ref{mainprop} and \ref{logprop}} \label{sect5} We first prove Proposition \ref{mainprop}, and then indicate the necessary changes for Proposition \ref{logprop} at the end of the section. In this proof, we use the notation  $(p_1)'=:q$.

 \subsection{Reductions}\label{ss51} By a limiting argument, it will suffice to replace $\mathbf{Qt}_1$ with a finite convex subset  $\S$. The constants implied by the almost inequality signs are not allowed to depend on $p_1$, $\S$ or $f_1,f_2,f_3$. Recall that we have the separation of scales \eqref{sepscl} for $\S$.
We   show how \eqref{mrwt} follows from
\begin{equation} \label{mainpf}
|\tr_\S(f_1,f_2,f_3)| \lesssim q \|f_1\|_{p_1}  |F_2|^{\frac{1}{p_2}}  , \qquad \forall\,   |f_j| \leq \cic{1}_{F_j}, \; j=2,3,
\end{equation}
whenever $  1\leq|G_3|<4$, which we summarize as $|G_3|\sim 1$ below.

Indeed, if $\lambda$ is  a power of $4$ chosen to have  $|\tilde{G_3}:= \{\lambda^{-1}x: x \in G_3\}|\sim 1$, then $|\tilde{F_2}:=\{\lambda^{-1}x: x \in F_2\}| \sim|G_3|^{-1}|F_2|$. Observe that  $\S_\lambda=\{s_\lambda=\lambda^{-1} I_s \times \lambda  \omega_s: s \in \S\}$ is still a convex subset of $\mathbf{Qt}_1$. If we employ \eqref{mainpf}, we get the existence of $\tilde{F_3} \subset \tilde{G_3},$ with $|\tilde{F_3}|\geq\frac18 |\tilde{G_3}|$ and
\begin{equation} \label{mainpf1}
\big|\tr_{\S_\lambda}\big(\mathrm{Dil}_{\frac1\lambda}^{p_1} f_1,\cic{1}_{\tilde{F_2}},\cic{1}_{\tilde{F_3}}\big)\big| \lesssim q \big\|\mathrm{Dil}^{p_1}_{\frac1\lambda}f_1\big\|_{p_1} |\tilde{F_2}|^{\frac{1}{p_2}}\sim q\|  f_1\|_{p_1} |{F_2}|^{\frac{1}{p_2}} |G_3|^{-\frac{1}{p_2}},
\end{equation}
It follows that $F_3 =\{\lambda y: y \in \tilde{F_3}\} \subset G_3$ is major in $G_3$.
Since by dilation invariance
$$
\big|\tr_{\S }(  f_1,\cic{1}_{ F_2},\cic{1}_{F_3})\big| = \big|\tr_{\S_\lambda}\big(\mathrm{Dil}_{\frac1\lambda}^{p_1} f_1, \lambda^{\frac{1}{p_2}}\cic{1}_{\tilde{F_2}}, \lambda^{ -\frac12}\cic{1}_{\tilde{F_3}}\big)\big| \sim |G_3|^{\frac{1}{p_2}-\frac12} \big|  \tr_{\S_\lambda}\big(\mathrm{Dil}_{\frac1\lambda}^{p_1} f_1,\cic{1}_{\tilde{F_2}},\cic{1}_{\tilde{F_3}}\big)\big|,
$$
\eqref{mrwt} follows by combining  \eqref{mainpf1} with the last display.  

We perform two more reductions. First, by linearity in the first function, it is enough to prove the case $\|f_1\|_{p_1}=1$ of \eqref{mainpf}. Furthermore, we can assume $|F_2| \leq |G_3|<4$. Indeed, if $|F_2| \geq |G_3|\geq 1$ instead, we take $r>p_2$: since the tuple $\big(\frac{1}{p_1},\frac1r,1-\frac{1}{p_1}-\frac{1}{r}\big)$ belongs to  the  range \eqref{hexa-tuples}
we apply the bound \eqref{BHTbdnonLor} (which holds for the discrete model $\Q_\S$ as well), and get
\begin{equation}
\label{red1}
|\tr_\S(f_1,f_2,f_3)| \lesssim q     \|f_1\|_{p_1}|F_2|^{\frac{1}{r}} |F_3|^{1-\frac{1}{p_1}-\frac1r}  \lesssim     |F_2|^{\frac{1}{r}}
\end{equation}
which is stronger than \eqref{mrwt} in the region $|F_2| \geq 1$. The dependence on $p_1$ of the constant in the above inequality follows easily if one tracks it down in the usual proof, given for instance (in the Fourier case) in Chapter 6 of \cite{ThWp}.

Summarizing, we have reduced to proving the following statement: show that for all $f_1$ of unit norm in $L^{p_1}(\R)$, $F_2 \subset \R$, $|F_2| \leq 4$, $G_3\subset \R$ with $|G_3|\sim 1$, there exists a major subset $F_3\subset G_3$ with
\begin{equation} \label{mainpf2}
|\tr_\S(f_1,f_2,f_3)| \lesssim q    |F_2|^{\frac{1}{p_2}}  , \qquad \forall\,   |f_j| \leq \cic{1}_{F_j}, \; j=2,3.
\end{equation}

\subsection{Exceptional sets}
We  define the   exceptional sets
\begin{equation} \label{eset1}
E_1 := \{\M_{p_1} f_1\geq c\}, \qquad E_2:= \{\M_{1}1_{F_2}\geq c |F_2|\}.
\end{equation}
For an appropriate choice of the constant $c$
\begin{equation} \textstyle \label{eset3}|E_1|\lesssim c^{-p_1}\|\M_{p_1} f_1\|_{p_1}^{p_1}\leq \frac14, \qquad |E_2|\lesssim c^{-1} \leq \frac14.
\end{equation}
Set $F_3:= G_3\backslash(E_1 \cup E_2)$; by the above, $|F_3| \geq \frac12 \geq \frac18 |G_3|$.
We note that when $|f_3|\leq \cic{1}_{F_3}$,
\begin{equation} \label{goodtiles}
\tr_\S(f_1,f_2,f_3) = \tr_{\S_\mathsf{good}}(f_1,f_2,f_3) , \qquad \S_\mathsf{good} = \{s \in \S : I_s \cap(E_1 \cup E_2)^c \neq \emptyset\}.
\end{equation}
Therefore we will just replace  $\S$ by $\S_\mathsf{good}$ in what follows. Thus, by construction, in view of the upper bound \eqref{ubsize},
\begin{equation}
\label{initsize}
\size_{f_1}(\S) \leq \sup_{s \in \S}\inf_{ I_s} \M_{p_1} f_1  \lesssim 1, \;\; \size_{f_2}(\S) \leq \sup_{s \in \S}\inf_{  I_s} \M_{1} \cic{1}_{F_2} \lesssim |F_2|, \;\;\size_{f_3}(\S) \leq 1.
\end{equation}
The last bound is a consequence of $f_3$ being uniformly bounded by 1.
\subsection{A multi-frequency Calder\'on-Zygmund decomposition}

Before we proceed with the proof, we present an abstract lemma which involves a multi-frequency Calder\'on Zygmund decomposition. This decomposition was proved  in \cite{NOT} for the Fourier case, and has since then  been successfully used in getting uniform estimates for the quartile operator \cite{OT} and  getting  bounds for the Walsh version of the lacunary Carleson's operator near $L^1$ \cite{DL}. Our lemma is essentially contained in \cite{OT}.
\begin{lemma} \label{CZ} Let $1<p_1<2$ , $q=(p_1)'$  and $\|f_1\|_{p_1}=1$. Define $\alpha=1 -\frac{2}{q}$ and let $\S$ be a finite forest with $\tops(\S)\leq A^2$ for some $A \geq 1$. Then there exists $g_1^\S$ such that
$$
\|g_{1}^\S\|_2 \lesssim   A^{\alpha}, \qquad \tr_{\S}(f_1, f_2,f_3)= \tr_{\S }(g_{1}^\S, f_2,f_3).
$$
\end{lemma}
\begin{proof}
Choose a  minimizing $\F$ and define the counting function $N  =\sum_{\T \in \F} \cic{1}_{I_\T}$.
Let $I \in \cic{I}$ be the collection of the children of the  maximal dyadic intervals inside  $E_1= \{\M_{p_1} f_1\ge c\}$. The intervals $I$ are pairwise disjoint, and
\begin{equation} \label{Cz1}
\sum_{I \in \cic{I}} |I| \leq |E_1| \leq 1, \qquad \| f_1\|_{L^{p_1}(I)} \lesssim 1.
\end{equation}
Note also that, as a consequence of \eqref{goodtiles},
$$
s \in \S,  I_s \cap I \neq \emptyset \implies I_s \supsetneq I^*,
$$
where $I^*$ is the dyadic parent of $I$.
This implies that the counting function $N$ is constant equal to $N_I$ on $I$.
We introduce the collection of tiles for each $I\in \cic{I}$
$$
\S_I = \{s \in \mathbf{Ti}: I_s = I, s \ll s'_j  \textrm{ for some } s' \in \S, \, j=1,\ldots,4\}.
$$
Observe that if $s \in \S_I$ and $s'$ is a quartile in $\S$ with $s\ll s'_1$, then $s \ll s'_j $ for $j=2,\ldots,4 $ too. This is because $|I_{s'}| \geq 4 |I|$, by the above observations. As a first consequence, $w_{s'_j}$ is a scalar multiple of $w_s$ on $I$, and  in particular
\begin{equation} \label{span}
w_{s'_j} \cic{1}_I \in H_{\S_I} \qquad \forall s' \in \S.
\end{equation}

 As a further consequence, $\#\S_I \leq N_I$, since each $\omega_s$ with $s \in \S_I$ contains the frequency component  of some of the $N_I$ trees whose top spatial interval contains $I$.
Let $v=\sum_{s \in \S_I} a_s w_s \in H_{\S_I}$. We see that, for all $q \geq 2$,
 \begin{equation} \label{gi4}
\|v\|_{L^q(I)} \leq \|v\|_{L^2(I)}^{\frac2q}  \|v\|_{L^\infty(I)}^{1-\frac2q} \leq \|v\|_{L^2(I)}^{\frac2q}\Big(  N_I^\frac12\textstyle\frac {\textstyle\big[\sum_{s \in \S_I} |a_s|^2\big]^\frac12}{\sqrt{|I|}} \Big)^{ 1-\frac2q} = N_I^{\frac12-\frac1q} \|v\|_{L^2(I)}.
\end{equation}
With this observation and the second half of \eqref{Cz1},
$$
|(f_1 ,v)_{L^2( I)}|   \leq   \|f_1\|_{L^{p_1}(I)} \|v\|_{L^{q}(I)} \lesssim N_I^{\frac\alpha2 }\|v\|_{L^2(I)} .
$$
By the Riesz representation theorem, we get the existence of  $g_I \in H_{\S_I}$, with \begin{equation} \label{gi}  \quad \|g_I\|_{L^2(I)}  \lesssim N_I^{\frac\alpha2}, \quad
(f_1,v)_{L^2(I)}= (g_I,v)_{L^2(I)} \; \forall v \in H_{\S_I}.  \end{equation}
We set
$$
g_{1}^\S = f_1 - \sum_{I \in \cic{I}}\big( f_1\cic{1}_I - g_I\big).
$$
In particular, $g_1^\S =f_1$ outside of $E_1$. Thus, for every $ s \in \S,$ $ j=1,\ldots,4$, in view of \eqref{span}
$$
\l f_1, w_{s_j} \r = \l g_1^\S \cic{1}_{E_1^c}, w_{s_j} \r + \sum_{I \in \cic{I}} \l f_1, w_{s_j}\cic{1}_I \r = \l g_1^\S \cic{1}_{E_1^c}, w_{s_j} \r + \sum_{I \in \cic{I}} \l g_I, w_{s_j}\cic{1}_I \r  = \l g_1^\S  , w_{s_j} \r,
$$and it follows immediately that
\begin{equation} \label{gi1}
\tr_{\S }(f_1,f_2,f_3) = \tr_{\S }(g_{1}^\S ,f_2,f_3),
\end{equation}
Since $|f_1| \lesssim 1$ on $E_1^c$ and $p_1<2,$
 \begin{equation} \label{gi2}
\|g_{1}^\S \cic{1}_{(E_1)^c}\|_{2}^2 = \int_{E_1^c} |f_1|^2 \, \d x \lesssim  \int_{E_1^c} | f_1|^{p_1} \, \d x \lesssim  \|f_1\|_{p_1}^{p_1} =1,
\end{equation}
while
 \begin{align} \label{gi3}
\|g_{1}^\S \cic{1}_{ E_1 }\|_{2}^2 &=\sum_{I \in \cic{I}}|I| \|g_I\|_{L^2(I)}^2  \lesssim\sum_{I \in \cic{I}}N_I^\alpha|I| = \int_{\cup_{I \in \cic{I}} I} N^\alpha    \\&   \leq\|N\|_{1}^\alpha \Big|\Big(\bigcup_{I \in \cic{I}} I\Big)\Big|^{1-\alpha} \leq A^{2\alpha}. \nonumber
\end{align}
In the last step,  we have used
  the bound on the sum of the $|I|$. Collecting \eqref{gi1}-\eqref{gi3}, we are done.
\end{proof}
\subsection{Proof of \eqref{mainpf2}}  Iterating the size Lemma \ref{sizelemma} with respect to $f_2$ and $f_3$ and then intersecting the forests, and recalling \eqref{initsize}, we decompose
$$
\S=   \bigcup \big\{\S_{n_2,n_3}: \; {2^{-n_2} \leq|F_2|^{\frac{1}{2}} }, 2^{-n_3} \leq 1\big\} $$ with each $\S_{n_2,n_3}$ being a convex forest such that
\begin{align*}&
\tops(\S_{n_2,n_3}) \lesssim \min\{ 2^{2n_2} ,2^{2n_3} \} ,\\ &\size_{f_1}(\S_{n_2,n_3}) \lesssim 1, \quad \size_{f_2}(\S_{n_2,n_3}) \lesssim 2^{-n_2}|F_2|^{\frac12} , \quad \size_{f_3}(\S_{n_2,n_3})) \leq 2^{-n_3}.
\end{align*}
We use Lemma \ref{CZ} for each  $\S_{n_2,n_3}$, obtaining  functions $g_{1}^{\S_{n_2,n_3}}$ satisfying\begin{equation}
\|g_{1}^{\S_{n_2,n_3}}\|_2 \lesssim \min\{2^{\alpha n_2}, 2^{\alpha n_3}\}
\end{equation}
and
 \begin{equation}
\tr_{\S_{n_2,n_3}}(f_1,f_2,f_3)=  \tr_{\S_{n_2,n_3}}\big(g_{1}^{\S_{n_2,n_3}},f_2,f_3\big).
\label{equality}
\end{equation}
Recall here that $\alpha= 1-\frac{2}{q}$.
Applying \eqref{sferef} of Lemma \ref{lemmasfe} (or the analogous estimate where $2$ and $3$ are permuted), we obtain the forest estimates
\begin{equation} \label{sfec}
 |\tr_{\S_{n_2,n_3}}(f_1, f_2,f_3)|=\big|  \tr_{\S_{n_2,n_3}}\big(g_{1}^{\S_{n_2,n_3}},f_2,f_3\big)\big| \lesssim |F_2|^{\frac12 }  \begin{cases} 2^{-n_3}2^{\alpha n_2} \displaystyle  &  2^{n_2} \leq 2^{n_3} \\ \displaystyle2^{-n_2}2^{\alpha n_3}   &  2^{n_3} \leq 2^{n_2} . \end{cases}
\end{equation}
We used that $\|f_2\|_2 \leq |F_2|^{\frac12} $ and $\|f_3\|_2 \leq |F_3|^{\frac12}\sim 1$. We now use the triangle inequality over the forests $\S_{n_2,n_3}$. This yields the estimate
\begin{equation}
 |\tr_{\S}(f_1, f_2,f_3)| \lesssim \sum_{\substack{2^{-n_2} \leq |F_2|^{\frac12} \\ 2^{-n_3} \leq 1}}  |\tr_{\S_{n_2,n_3}}(f_1, f_2,f_3)|.
\end{equation}
We split the above sum into two separate regimes.
\vskip0.7mm \noindent $\bullet$ \textsc{Regime} $R_2=\{2^{-n_2} \leq |F_2|^{\frac12}, 1 \leq 2^{n_3} \leq 2^{n_2}\}$.
\vskip0.7mm \noindent
We use the second forest estimate in \eqref{sfec}, and obtain
\begin{align*}
\sum_{(n_2,n_3) \in R_2 }   |\tr_{\S_{n_2,n_3}}(f_1, f_2,f_3)| \
 &\lesssim |F_2|^{\frac12} \sum_{2^{-n_2} \leq |F_2|^{\frac12}} 2^{-n_2}\sum_{1\leq 2^{n_3} \leq 2^{n_2}} 2^{\alpha n_3}  \\ & \lesssim  |F_2|^{\frac12} \sum_{2^{-n_2} \leq|F_2|^{\frac12}} 2^{-n_2(1-\alpha)}   \\ & \lesssim   q  |F_2|^{\frac12+\frac{1}{q}} =q |F_2|^{\frac{1}{p_2}}. \end{align*}
Here we used that
\begin{equation}
\sum_{n \geq n_0} 2^{-(1-\alpha) n} \sim   \frac{1}{1-\alpha} 2^{-(1-\alpha) n_0} \sim q |F_2|^{\frac{1}{q}},
\end{equation} for $n_0 = -\frac12 \log |F_2|$ and that
\begin{equation} \textstyle
 \frac{1}{2}+ \frac{1}{q} = \frac32 -\frac{1}{p_1}= \frac{1}{p_2}.
\end{equation}
\vskip0.7mm \noindent $\bullet$ \textsc{Regime} $R_3=\{  |F_2|^{-\frac12} \leq 2^{n_2} \leq 2^{n_3}\}$.
\vskip0.7mm \noindent
We use the first forest estimate in \eqref{sfec}, and get
\begin{align*}
\sum_{(n_2,n_3) \in R_3 }   |\tr_{\S_{n_2,n_3}}(f_1, f_2,f_3)| \
 &\lesssim |F_2|^{\frac12} \sum_{2^{-n_3} \leq |F_2|^{\frac12}} 2^{-n_3}\sum_{|F_2|^{-\frac12} \leq 2^{n_2} \leq 2^{n_3}} 2^{\alpha n_2}  \\ & \lesssim  |F_2|^{\frac12} \sum_{2^{-n_3} \leq  |F_2|^{\frac12}} 2^{-n_3(1-\alpha)}   \\ & \lesssim     q|F_2|^{\frac12+\frac{1}{q}} = q|F_2|^{\frac{1}{p_2}}.
 \end{align*}
Combining the last two displays, we obtain exactly \eqref{mainpf2}, and in turn, the theorem.
\subsection{Proof of Proposition \ref{logprop}}
By the standard argument that converts generalized restricted weak type estimates into weak type (see the proof of \eqref{mrwtmain} given \eqref{mrwt}), \eqref{wtelog} follows if we show that  for all subsets $\S\subset \mathbf{Qt}$,   $f_1 \in L^p(\R)$, $f_2 \in L^2(\R),$ and $G_3\subset \R$ there exists a major subset $F_3 \subset G_3$ such that for all functions $f_3$ with  $|f_3|\leq \cic{1}_{F_3}$ \begin{equation} \label{rwtlog}
|\tr_\S(f_1,f_2,f_3)| \lesssim p' \|f_1\|_p \| f_2\|_2 |G_3|^{1-\frac1r}, \qquad 1<p<2.
\end{equation}
Performing the same reductions as in Subsection \ref{ss51} (i.e.\ using linearity and dyadic scaling) it further suffices to assume $\|f_1\|_p=\|f_2\|_2=1$, $ 1 \leq |G_3|<4$. 

The first step is yet again eliminating the two exceptional sets
$$
|E_1 := \{\M_p f_1 \geq c  \}| \leq \frac14   \qquad |E_2:= \{\M_2 f_2 \geq c  \}| \leq \frac14
$$
with $c>0$ suitably chosen, so that  $F_3:= G_3 \backslash(E_1 \cup E_2)$ is a major subset of $G_3$. Arguing just like we did at the end of Subsection 5.2, we arrive at the  following starting point for the size decomposition:
\begin{equation}
\label{initsize2}
\size_{f_1}(\S)\leq c, \qquad \size_{f_2}(\S) \leq   c, \qquad \size_{f_3}(\S) \leq 1.
\end{equation}
We now decompose $\S$ by means of the size lemma with respect to   $f_2,f_3$:
$$
\S=   \bigcup \big\{\S_{n_2,n_3}: \; 2^{-n_2}, 2^{-n_3} \leq 1\big\} $$ with
\begin{align*}&
\tops(\S_{n_2,n_3})=\sum_{\T\in \F_{n_2,n_3}}|I_\T| \lesssim 2^{2n_\circ} ,\quad n_\circ:=\min\{n_2,n_3\},\\ &\size_{f_1}(\S_{n_2,n_3}) \lesssim 1, \quad \size_{f_2}(\S_{n_2,n_3}) \lesssim 2^{-n_2}, \quad \size_{f_3}(\S_{n_2,n_3}) \leq 2^{-n_3}.
\end{align*} Applying Lemma \ref{CZ} to $f_{1}$ relatively to each collection $\S_{n_{2},n_{3}}$, we get  the existence of a function $g_{1}^{\S_{n_2,n_3}}$ with $$
\|g_{1}^{\S_{n_2,n_3}} \|_{2} \lesssim 2^{\left(1-\frac{2}{p'}\right)n_\circ}, $$ and $\tr_{\S_{n_2,n_3}}(f_1,f_2,f_3)=  \tr_{\S_{n_2,n_3}}\big(g_{1}^{\S_{n_2,n_3}},f_2,f_3\big)$.
Applying Lemma \ref{lemmasfe}, the forest estimate then becomes (denoting $n_\star:=\max\{n_2,n_3\}$)
$$
|\tr_{\S_{n_2,n_3}}(f_1, f_2,f_3)| \lesssim \size_{f_\star}(\S_{n_2,n_3}) \|g_{1}^{\S_{n_2,n_3}}\|_2 \|f_\circ\|_2 \lesssim  2^{-n_\star} 2^{\left(1-\frac{2}{p'}\right)n_\circ}.
$$
We have tacitly used that $\|f_2\|_2=1, 
\|f_3\|_2\leq |F_3|^{\frac12}\lesssim 1.$
Finally, we use the triangle inequality to estimate
\begin{align*}
|\tr_{\S}(f_1, f_2,f_3)| & \leq \Big(\sum_{1 \leq 2^{n_2} \leq 2^{n_3}} + \sum_{1 \leq 2^{n_3} \leq 2^{n_2}} \Big) |\tr_{\S_{n_2,n_3}}(f_1, f_2,f_3)| \\ & \lesssim \sum_{2^{n_2} \geq 1} 2^{-n_2}  \sum_{1 \leq 2^{n_3} \leq 2^{n_2}} 2^{n_3 \left(1-\frac{2}{p'}\right)} + \sum_{2^{n_3} \geq 1} 2^{-n_3} \sum_{1 \leq 2^{n_2} \leq 2^{n_3}} 2^{n_2\left(1-\frac{2}{p'}\right)} \\   & \lesssim \sum_{  2^{n_2}\geq 1} 2^{-\frac{2}{p'}n_2} + \sum_{  2^{n_3}\geq 1} 2^{-\frac{2}{p'}n_3} \lesssim p',
\end{align*}which yields  (the scaled version of) \eqref{rwtlog}
The proof of \eqref{wtelog} is thus complete.

\section{Concluding remarks}
There are two immediate directions for improvement of the results contained in Section \ref{main}. The first one is of technical nature, and consists in extending all the results here to the Fourier analogue of the quartile operator, and thus ultimately to $\mathsf{BHT}$. The main difficulty lies in recasting efficiently the multi-frequency Calder\'on-Zygmund decomposition of Lemma \ref{CZ} in the Fourier case: while suitable analogues of \eqref{gi4} hold true, the insertion of the phase space projection as in \eqref{equality} creates a nontrivial error term which must be carefully estimated. We plan on dealing with this issue in a future work.

A second, probably   harder and more interesting task is upgrading Theorem \ref{lorentzthm} to the full strength of the conjectured \eqref{cjt}. The techniques  employed here do not seem to work when both functions in \eqref{mrwtmain} are taken to be unrestricted. This case is intrinsically more difficult, since two of the three functions in $\tr_\S$ are not in $L^2$ locally and  thus they both  need to be treated by the multi-frequency Calder\'on-Zygmund decomposition. A very appealing situation  is the  symmetric case
\begin{equation}
\label{polar}
\V_\S: L^{\frac43}(\R) \times L^{\frac43}(\R)  \to L^{\frac23,\infty}(\R).
\end{equation}
To prove this bound, it will be enough to show that for any $F\subset \R$ with $|F|=1$ and $\|f\|_{\frac43}=1$ there exists a major subset $G$ of $F$ such that
\begin{equation}
\label{polar1}
|\tr_\S (f,f,g)| \lesssim 1,\text{ when } |g| \leq \cic{1}_{G}.
\end{equation}
Indeed, \eqref{polar} for arbitrary $f_1,f_2$ will then follow by applying \eqref{polar1} to $f\in\{f_1,f_2,f_1-f_2\}$. It is tempting to believe that \eqref{polar1} should be the easiest case since the same multi-frequency Calder\'on-Zygmund decomposition would simultaneously resolve both components.

\subsection*{Acknowledgments} We thank Dimitry Bilyk for valuable discussions and pointers to preexisting literature. We would also like to thank the very careful referee who has pointed out some inaccuracies in the first version of the manuscript.

\bibliography{12100}{}
\bibliographystyle{amsplain}
\end{document}